\documentclass[a4paper,11pt]{article}
\usepackage[ngerman, english]{babel}
\usepackage[T1]{fontenc}
\usepackage[utf8]{inputenc}
\usepackage{amsmath}
\usepackage{amssymb}
\usepackage{amsthm}
\usepackage{graphicx}
\usepackage{here}
\usepackage{color}
\usepackage{xcolor}
\usepackage{enumerate}
\usepackage{lmodern}
\usepackage{fancyvrb}
\usepackage[plainpages=false]{hyperref}
\usepackage{caption}
\usepackage{subfigure}
\usepackage{epstopdf}
\theoremstyle{plain}

\newtheorem{defi}{Definition}[section]
\newtheorem{theo}[defi]{Theorem}

\newtheorem{con}[defi]{Conjecture}
\newtheorem{obs}[defi]{Observation}

\newtheorem*{theo*}{Theorem}

\newcounter{claimcount}
\newenvironment{claim}{\refstepcounter{claimcount}\textbf{Claim \arabic{claimcount}.}}{}

\usepackage{etoolbox}
\AtBeginEnvironment{proof}{\setcounter{claimcount}{0}}

\theoremstyle{remark}

\newcommand{\ENDproof}{\hfill $\blacksquare$\medskip\par}

\title{A counterexample to the $S_{10}$- and the $S_{12}$-Conjecture}

\author{Isaak H.~Wolf}
	
\date{}

\begin{document}

\maketitle

\begin{abstract}
For two graphs $G$ and $H$, a mapping $f\colon E(G) \to E(H)$ is an $H$-coloring of $G$, if it is a proper edge-coloring and for every $v \in V(G)$ there exists a vertex $u \in V(H)$ with $f(\partial_G(v))=\partial_H(u)$. Motivated by the Petersen Coloring Conjecture, Mkrtchyan [A remark on the Petersen coloring conjecture of Jaeger, \emph{Australas. J. Combin.}, 56 (2013), 145-151] and Mkrtchyan together with Hakobyan [$S_{12}$ and $P_{12}$-colorings of cubic graphs, \emph{Ars Math. Contemp.}, 17 (2019), 431-445] made the following two conjectures. (I) Every cubic graph has an $S_{10}$-coloring, where $S_{10}$ is a graph on 10 vertices sometimes also referred to as the Sylvester graph. (II) Every cubic graph with a perfect matching has an $S_{12}$-coloring, where $S_{12}$ is the graph obtained from $S_{10}$ by replacing the central vertex with a triangle. In this note we present a (rather small) counterexample to both conjectures.
\end{abstract}


{\bf Math. Subj. Class.:} 05C15, 05C70.

\section{Introduction}
In this note we consider finite graphs that may have parallel edges but no loops. For two graphs $G$ and $H$, an \emph{$H$-coloring} of $G$ is a mapping $f\colon E(G) \to E(H)$ such that
\begin{itemize}
	\item if $e_1,e_2 \in E(G)$ are adjacent, then $f(e_1) \neq f(e_2)$,
	\item for every $v \in V(G)$ there exists a vertex $u \in V(H)$ with $f(\partial_G(v))=\partial_H(u)$.
\end{itemize}

If such a mapping exists, then we write $H \prec G$ and say $H$ \emph{colors} $G$. In 1988 Jaeger made the following seminal conjecture, where $P$ denotes the Petersen graph (see Figure~\ref{fig:PS_10S_12}).

\begin{con}[Petersen Coloring Conjecture, Jaeger \cite{jaeger1988nowhere}, 1980]
	\label{con:P_con}
If $G$ is a bridgeless cubic graph, then $P \prec G$.	
\end{con}
 
If this conjecture is correct, then some other long-standing conjectures such as the Berge-Fulkerson Conjecture \cite{fulkerson1971blocking} and the 5-Cycle Double Cover Conjecture (see \cite{C.-Q._Zhang_book}) are also true. Duo to its far reaching consequences not only for cubic graphs, the Petersen Coloring Conjecture can be considered as one of the most important conjectures in graph theory.
Conjecture~\ref{con:P_con} is trivially true for cubic graphs with chromatic index 3 and it is verified for all bridgeless cubic graphs of order at most 36 with the help of a a computer \cite{brinkmann2013generation}. Nevertheless, a general answer seems to be far away. The Petersen Coloring Conjecture motivated research in several directions. One line of research is to use other graphs for coloring and study whether for different graph classes there exists a graph (or a set of graphs) that colors all graphs from this class. For instance, in \cite{MTZ_r_graphs} and \cite{ma2025sets}  this question is studied for the class of $r$-regular graphs and of $r$-graphs, respectively, for all $r>3$. For cubic graphs there are the following three conjectures, where the latter two are for cubic graphs with bridges.

\begin{con}[$S_{4}$-Conjecture, Mazzuoccolo \cite{MAZZUOCCOLO2013235} (see also \cite{MZ_S4equi}), 2013]
	\label{con:S_4_con}
	If $G$ is a bridgeless cubic graph, then $S_4 \prec G$.
\end{con}

\begin{con}[$S_{10}$-Conjecture, Mkrtchyan \cite{Mkrtchyan_Pet_col}, 2012]
	\label{con:S_10_con}
If $G$ is a cubic graph, then $S_{10} \prec G$.	
\end{con}

\begin{con}[$S_{12}$-Conjecture, Mkrtchyan and Hakobyan \cite{Hakobyan2019S12AP}, 2019]
	\label{con:S_12_con}
If $G$ is a cubic graph with a perfect matching, then $S_{12} \prec G$.	
\end{con}

The graphs $S_4$, $S_{10}$ and $S_{12}$ are depicted in Figure~\ref{fig:PS_10S_12} and are the only graphs that may fulfill the statements of the conjectures above in the following sense. If $H$ is a connected graph that colors every bridgeless cubic graph, then either $H$ is isomorphic to $P$, or $H$ contains $S_4$ as an induced subgraph  \cite{MTZ_r_graphs}. If $H$ is a connected graph that colors every cubic graph, then $H$ is isomorphic to $S_{10}$; if $H$ is a connected graph that colors every cubic graph with a perfect matching, then $H$ is isomorphic to either $S_{10}$ or $S_{12}$ \cite{MTZ_r_graphs}. The $S_4$-Conjecture, now a theorem, was verified by Kardo{\v{s}}, M{\'a}{\v{c}}ajov{\'a} and Zerafa \cite{KARDOS20231}.
In this short note we give an answer to the $S_{10}$- and the $S_{12}$-Conjecture by constructing a cubic graph with a perfect matching that can not be colored by $S_{10}$. Note that the relation $\prec$ is transitive, which imply that the presented graph can also not be colored by $S_{12}$. Hence, both the $S_{10}$- and the $S_{12}$-Conjecture are false.

\begin{figure}[htbp]
		\centering
		\subfigure[$P$]{
			\begin{minipage}[t]{0.28\textwidth}
				\centering
				\includegraphics[height=4.5cm]{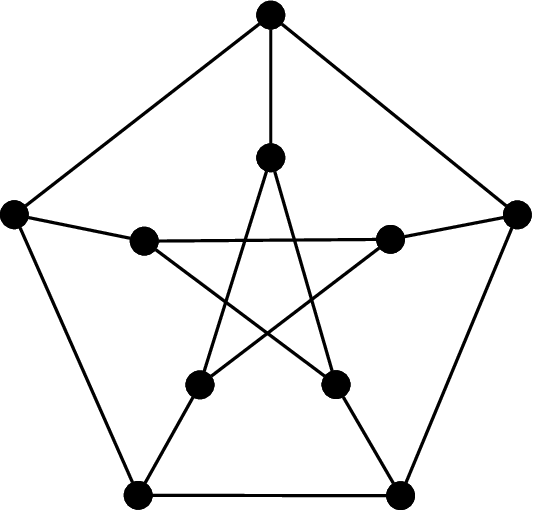}
			\end{minipage}
		}
		\subfigure[$S_{4}$]{
			\begin{minipage}[t]{0.15\textwidth}
				\centering
				\includegraphics[height=2cm]{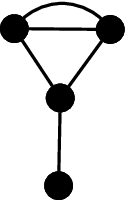}
			\end{minipage}
		}
		\subfigure[$S_{10}$]{
			\begin{minipage}[t]{0.22\textwidth}
				\centering
				\includegraphics[height=3cm]{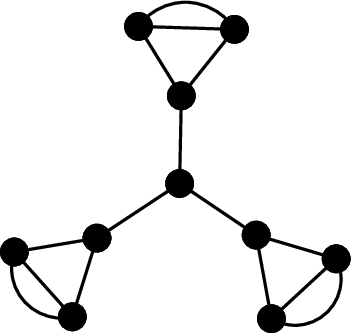}
			\end{minipage}
		}
		\subfigure[$S_{12}$]{
			\begin{minipage}[t]{0.22\textwidth}
				\centering
				\includegraphics[height=3cm]{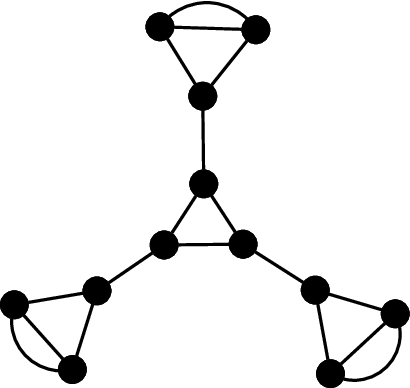}
			\end{minipage}
		}
		\caption{The Petersen graph and the graphs $S_4$, $S_{10}$ and $S_{12}$.}
		\label{fig:PS_10S_12}
	\end{figure}

\section{Definitions and preliminary results}
A \emph{circuit} is a connected 2-regular graph; a circuit is \emph{even} if it is of even order; a \emph{$k$-circuit} is a circuit of order $k$.

Let $G$ be a graph. A \emph{matching} is a set $M\subseteq E(G)$ such that no two edges of $M$ are adjacent. Moreover, $M$ is \emph{perfect} if every vertex of $G$ is incident with an edge of $M$. The \emph{chromatic index}, denoted $\chi'(G)$, is the smallest integer $k$ for which $E(G)$ can be partitioned into $k$ matchings.
For a set $X \subseteq V(G)$, the set of edges with exactly one end in $X$ is denoted by $\partial_G(X)$.
Let $E' \subseteq E(G)$. We say that $E'$ \emph{induces} a subgraph $G'$ of $G$ if $E(G') = E'$ and $V(G')$ contains all vertices of $G$ that are incident with an edge of $E'$. Such a subgraph $G'$ is denoted by $G[E']$. 

We will use the following basic observation concerning $H$-colorings.

\begin{obs}
	\label{obs:coloring_basics_new}
	Let $H$ and $G$ be graphs, let $f \colon E(G) \to E(H)$ be an $H$-coloring of $G$ and let $H'$ be a subgraph of $H$.
	\begin{itemize}
		\item[(i)] If $H'$ is $k$-regular, then $f^{-1}(E(H'))$ induces a $k$-regular subgraph in $G$.
		\item[(ii)] $\chi'(G[f^{-1}(E(H'))]) \leq \chi'(H')$.
		\item[(iii)] If $e$ is a bridge of $G$, then $f(e)$ is a bridge of $H$.
	\end{itemize}
\end{obs}

\begin{proof}
Statement $(i)$ is a direct consequence of the definition of $H$-colorings; statement $(iii)$ has been proven in \cite{Hakobyan2019S12AP}. 
By definition of the chromatic index, $E(H')$ can be partitioned into $\chi'(H')$ matchings. By $(i)$, the pre-image of them are $\chi'(H')$ pairwise disjoint matchings in $G$. Hence, $f^{-1}(E(H'))$ can be partitioned into $\chi'(H')$ matchings, which is equivalent to statement $(iii)$.
\end{proof}

\section{A cubic graph with a prefect matching that cannot be colored by $S_{10}$}
In this section we construct a cubic graph $G^*$ that has a perfect matching but cannot be colored by $S_{10}$.

Take a copy of $P$. For one vertex, subdivide the edges incident to it and expand one of the new vertices to a triangle. Attach a copy of $S_4$ to each of the three vertices of degree 2. We obtain a cubic graph $G^*$ that has a perfect matching (see Figure~\ref{fig:G}).

\begin{figure}[H]
\centering
\includegraphics[height=7cm]{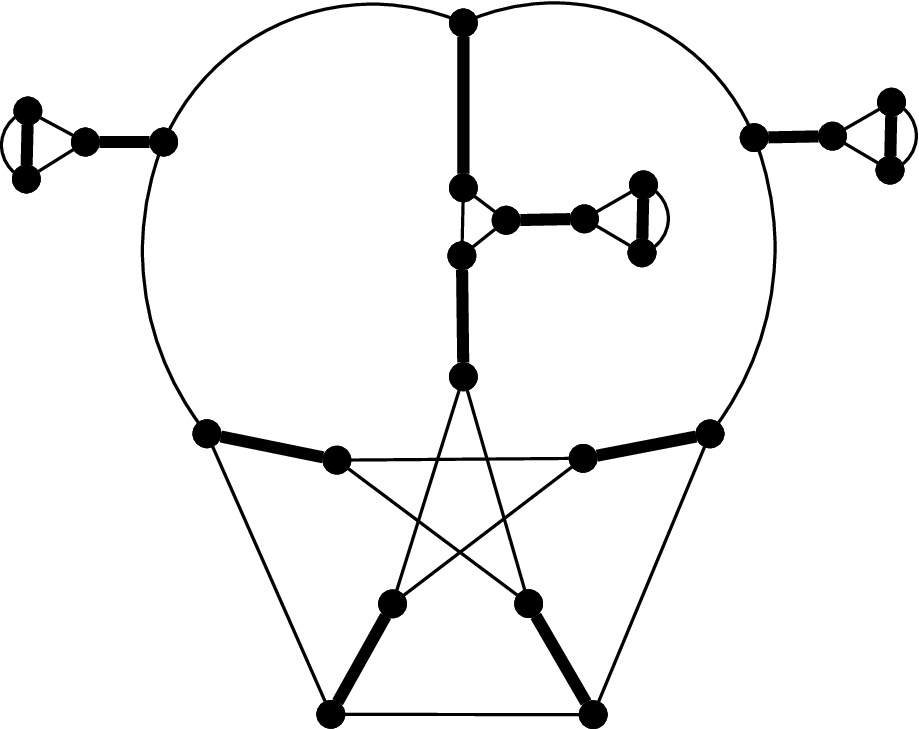}
\caption{The graph $G^*$ with a perfect matching (bold edges).}
\label{fig:G}
\end{figure}

We now prove that it does not admit an $S_{10}$-coloring.

\begin{theo}
	\label{thm:not_colorable_with_S_10}
	$G^*$ can not be colored by $S_{10}$.
\end{theo}
	
\begin{proof}
Suppose $f:E(G^*) \to E(S_{10})$ is an $S_{10}$-coloring of $G^*$. 
Let $H$ be the subgraph of $G^*$ isomorphic to the graph obtained from $P$ by deleting one vertex. Let $E_1 = E(H)$ and let $E_2$  be the set of remaining edges of $G^*$ that do not belong to a copy of $S_4$.
Moreover, let $A_1 \subset E(S_{10})$ be the set of edges of $S_{10}$ belonging to a 2-circuit and let $A_2 = E(S_{10}) \setminus A_1$. In the following we will refer to the vertices of $G^*$ and $S_{10}$ by the labels introduced in Figure~\ref{fig:edge-set_partitions}; in this figure also the edge sets $E_1, E_2, A_1$ and $A_2$ are depicted.

\begin{figure}[htbp]
		\centering
		\subfigure[$G^*$ and edge sets $E_1$ (red), $E_2$ (green).]{
			\begin{minipage}[t]{0.55\textwidth}
				\centering
				\includegraphics[height=7cm]{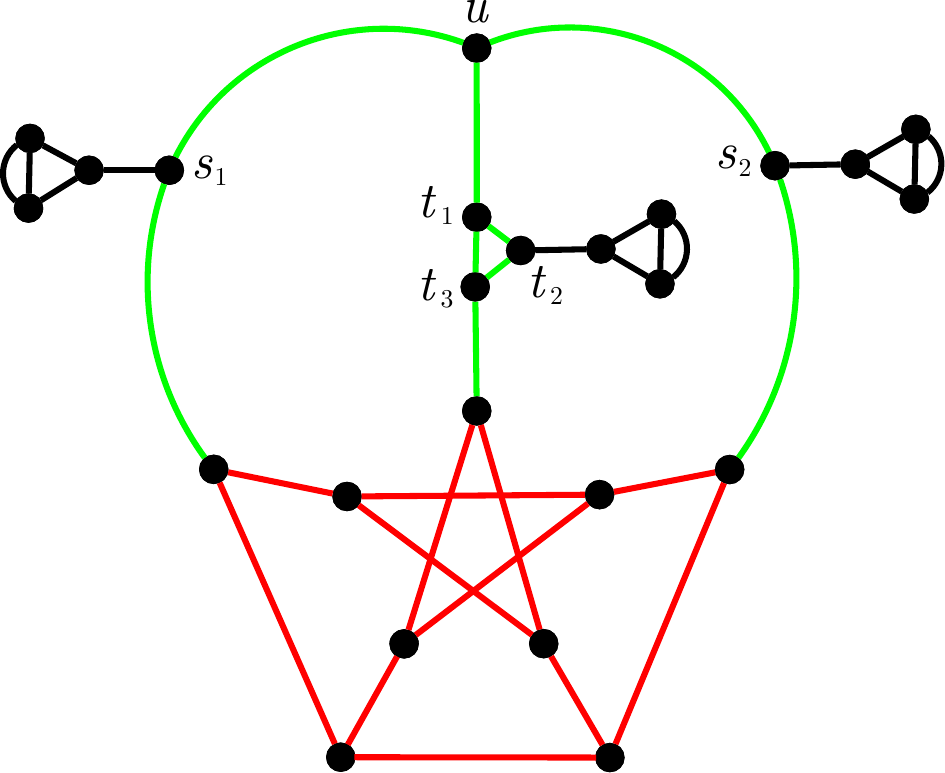}
			\end{minipage}
		}
		\subfigure[$S_{10}$ and edge sets $A_1$ (red), $A_2$ (green)]{
			\begin{minipage}[t]{0.35\textwidth}
				\centering
				\includegraphics[height=4cm]{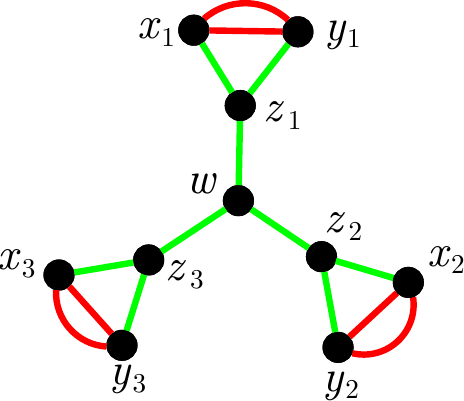}
			\end{minipage}
		}
		\caption{The labels and edge sets of $G^*$ and $S_{10}$ used in the proof of Theorem~\ref{thm:not_colorable_with_S_10}.}
		\label{fig:edge-set_partitions}
	\end{figure}


\begin{claim}
$f^{-1}(A_1) \cap E_2 = \emptyset$.
\label{clm:A_1_E_2}
\end{claim}\\
\emph{Proof of Claim \ref{clm:A_1_E_2}.}
By statement $(i)$ of Observation~\ref{obs:coloring_basics_new}, the subgraph of $G^*$ induced by $f^{-1}(A_1)$ consists of pairwise disjoint circuits, which are even circuits by statement $(ii)$. Moreover, by statement $(iii)$ of Observation~\ref{obs:coloring_basics_new}, no edge of these even circuits is adjacent with a bridge of $G^*$. Hence, by the structure of $G^*$ no edge of $E_2$ is mapped to an edge of $A_1$.
\ENDproof

Let $H'$ be the subgraph of $H$ induced by $f^{-1}(A_1) \cap E_1$

\begin{claim}
$H'$ is either a $6$-circuit or an $8$-circuit.
\label{clm:A_1_E_1}
\end{claim}\\
\emph{Proof of Claim \ref{clm:A_1_E_1}.}
By Claim~\ref{clm:A_1_E_2}, $H'$ consists of pairwise disjoint even circuits. Furthermore, observe that $\chi'(H)=4$ and $\chi'(S_{10}[A_2])=3$. Hence, by (ii) of Observattion~\ref{obs:coloring_basics_new}, atleast one edge of $E_1$ is mapped to an edge in $A_1$, i.e. $E(H_1) \neq \emptyset$. Since $H$ is a subgraph of $P$, the only even circuits it contains are 6- and 8-circuits, which implies that $H'$ is either a $6$- or an $8$-circuit.
\ENDproof

Without loss of generality we assume that the edges of $H'$ are alternately mapped to the parallel edges connecting $x_1$ and $y_1$.
Let $f_V\colon V(G^*) \to V(S_{10})$ be the mapping induces by $f$, i.e. for every $v \in V(G^*)$, the vertex $f_V(v)$ is the unique vertex $v' \in V(S_{10})$ with $f(\partial_{G^*}(v))=\partial_{S_{10}}(v')$.
\begin{claim}
$f_V(V(H)) \subseteq \{x_1,y_1,z_1\}$.
\label{clm:f_V(H)}
\end{claim}\\
\emph{Proof of Claim \ref{clm:f_V(H)}.}
Suppose there is a vertex $v \in V(H)$ that is not mapped to $x_1$, $y_1$ or $z_1$.
Each vertex of $H'$ is mapped to either $x_1$ or $y_1$, and hence every vertex adjacent to a vertex of $H'$ is mapped to a vertex in $\{x_1,y_1,z_1\}$. By the structure of $H$, we deduce that $H'$ is a 6-circuit, $v$ is of degree 2 in $H$ and both its neigbhours are adjacent to two vertices in $H'$. Hence, $f_V(v)$ is incident with two edges that are also incident with $x_1,y_1$ or $z_1$, a contradiction.
\ENDproof

\begin{claim}
$f_V(v_1) \in \{z_1,w\}$.
\label{clm:f_V(u)}
\end{claim}\\
\emph{Proof of Claim \ref{clm:f_V(u)}.}
By Claim~\ref{clm:A_1_E_2}, $f_V(u) \in \{w, z_1, z_2, z_3\}$. Hence, by symmetry suppose that $f_V(u)=z_2$. As a consequence, $f(us_1) \in \{z_2x_2, z_2y_2\}$ or $f(us_2) \in \{z_2x_2, z_2y_2\}$; without loss of generality we assume $f(us_1) = z_2x_2$.
Hence, $f_V(s_1) \in \{x_2,z_2\}$, which contradicts Claim~\ref{clm:f_V(H)}.
 \ENDproof

By Claim~\ref{clm:f_V(u)}, one edge incident with $u$ is mapped to $z_1w$.
If $f(us_1)=z_1w$, then either $f_V(s_1)=z_1$, in contradiction to (iii) of Observation~\ref{obs:coloring_basics_new}, or $f_V(s_1) = w$, in contradiction to Claim~\ref{clm:f_V(H)}.
Thus, by symmetry we may assume $f(ut_1)= z_1w$. Observe that $t_1t_2, t_2t_3, t_1t_3$ are mapped to three mutually adjacent edges. Thus, Claim~\ref{clm:A_1_E_2} implies $f(\{t_1t_2, t_2t_3, t_1t_3\})=\partial_{S_{10}}(z_1)$ or $f(\{t_1t_2, t_2t_3, t_1t_3\})=\partial_{S_{10}}(w)$. In the first case we deduce $f(\partial_{G^*}(\{t_1,t_2,t_3\}))=\partial_{S_{10}}(z_1)$ in contradiction to (iii) of Observation~\ref{obs:coloring_basics_new}; in the second case we deduce $f(\partial_{G^*}(\{t_1,t_2,t_3\}))=\partial_{S_{10}}(w)$ in contradiction to Claim~\ref{clm:f_V(H)}.


\end{proof}

\bibliography{Lit_reg_graphs}{}
\addcontentsline{toc}{section}{References}
\bibliographystyle{abbrv}

\end{document}